\documentclass[12pt,reqno]{article}
\usepackage{amssymb, amsmath}
\usepackage{amssymb,amsthm}
\newtheorem{thm}{Theorem}

\numberwithin{Definition}{section}
\numberwithin{Lemma}{section}
\numberwithin{Proposition}{section}

\numberwithin{Corollary}{section}
\numberwithin{Example}{section}
\numberwithin{Remark}{section}
\begin{document}
\begin{center}
\textbf{ASYMPTOTIC EXPANSION OF THE WAVELET TRANSFORM WITH ERROR TERM}
\end{center}
\begin{center}
R S Pathak and Ashish Pathak
\end{center}
\begin{center}
\textbf{Abstract }
\end{center}
Using Wong's technique asymptotic expansion for the wavelet transform is derived and thereby asymptotic expansions for Morlet wavelet transform, Mexican Hat wavelet transform and Haar wavelet transform are obtained.
\footnote{ This work is contained in the research monograph " The Wavelet Transform" by Prof. R S Pathak and edited by Prof. C. K. Chui (Stanford University, U.S.A.) and published by Atlantis Press/World Scientific (2009), ISBN: 978-90-78677-26-0, pp:154-164}
\section{Introduction}The wavelet transform of $ f $ with respect to the wavelet $ \psi $ is defined by
\begin{equation}
\label{eq:Int-510.2}
 (W_\psi f)(b,a) = \frac{1}{\sqrt{a}}\int_{-\infty}^\infty f(t)
  \overline{\psi\left(\frac{t-b}{a}\right)}\,dt,b\in \mathbb{R},a>0,
\end{equation}
provided the integral exists~\cite{Debnath}. Using Fourier transform it can also be expressed as
\begin{equation}
\label{eq:Int-510.3}
(W_\psi f)(b,a) = \frac{\sqrt{a}}{2\pi}\int_{-\infty}^\infty \hat f(\omega)e^{ib\omega}
  \overline{\hat \psi}(a\omega)d\omega,
\end{equation}
where
\begin{equation}
\nonumber \hat f(\omega)= \int_{-\infty}^\infty e^{-it\omega} f(t)dt.
\end{equation}
Asymptotic expansion with explicit error term for the general integral
\begin{equation}
\label{eq:Int-510.4}
I(x)=\int_0^\infty f(t)h(x t) dx,
\end{equation}
where $h(t)$ is an oscillatory function, was obtained by Wong \cite{Wong},
\cite{Rwong} under different conditions on $g$ and $h$. Then the asymptotic
expansion for (\ref{eq:Int-510.3}) can be obtained by setting $g(t) =
e^{ibt}\hat{f}(t)$ for fixed $b \in \mathbb{R}$. Let us recall basic
results from \cite{Rwong} which will be used in the present investigation.
\vspace{.5pc} Here we assume that $g(t)$ has an expansion of the form
\begin{eqnarray}
\label{eq:Int-151.5}
    g(t) \nonumber &\sim & \sum^{\infty}_{s=0}c_s t^{s+\lambda-1} \;\;\; \textrm{as}\;\;
    t\rightarrow 0,\\
    &=&\sum^{n-1}_{s=0}c_s t^{s+\lambda-1}+g_n(t)
\end{eqnarray}
where $0<\lambda \leq 1.$ Regarding the function $h$, we assume that
as $t\rightarrow 0+,$
\begin{eqnarray}
\label{eq:Int-152.5}
h(t) = O(t^\rho), \;\;\; \rho + \lambda >0,
\end{eqnarray}
and that as $t\rightarrow +\infty,$
\begin{eqnarray}
\label{eq:Int-153.5}
h(t) \sim \exp(i \tau t^p)\sum^{\infty}_{s=0} b_s t^{-s-\beta},
\end{eqnarray}
where $\tau \neq 0$ is real, $p \geq 1$ and $0<\beta \leq 1.$ Let
$M[h; z]$ denote the generalized Mellin transform of $h$ defined by
\begin{eqnarray}
\label{eq:Int-154.5}
M[h; z]=\lim_{\varepsilon \rightarrow \, 0+}\int^{\infty}_{0}
t^{z-1}h(t) \exp (-\varepsilon t^p) dt.
\end{eqnarray}
This, together with  (\ref{eq:Int-196.5}) and \cite[p.216]{Rwong}, gives
\begin{eqnarray}
\label{eq:Int-155.5}
I(x) = \sum^{n-1}_{s=0}c_s M[h;
s+\lambda]x^{-s-\lambda}+\delta_n(x),
\end{eqnarray}
where
\begin{eqnarray}
\label{eq:Int-156.5}
\delta_n(x)=\lim_{\varepsilon \rightarrow
0+}\int^{\infty}_{0}g_n(t)h(xt)\exp(-\varepsilon t^p)dt.
\end{eqnarray}
If we now define recursively $h^\circ (t)=h(t)$  and
\begin{eqnarray*}
h^{(-j)}(t)=-\int^{\infty}_{t}h^{(-j+1)}(u) du, \;\;\; j=1,2,...,
\end{eqnarray*}
Repeated integration by part, we have
\begin{eqnarray}
\label{eq:Int-210.5}
h^{(-j)}(t)\sim exp(i\tau t^p) \sum_{s=0}^\infty b_s^{(j)} t^{-\mu_{s,j}}, \,\,\,\,\,\,\,\, as \,\,\,\, \rightarrow \infty,
\end{eqnarray}
where $b_s^{(j)}$ are some constants and for each $j$, and ${\mu_{s,j}}$ is a monotonically increasing sequence of positive numbers depending on $p$ and $ \beta $.\\
Then conditions of validity of aforesaid results are given by the
following \cite[Theorem 6, p.217]{Rwong}:
\begin{thm}
\label{thm:ch1sec2-2.4}.
Assume that (i)
$g^{(m)}(t)$ \textit{is continuous on} $(0,\infty)$, \textit{where}
$m$ \textit{is a non-negative integer}; (ii) $g(t)$ \textit{has an
expansion of the form} (\ref{eq:Int-151.5}), \textit{and the expansion is} $m$
\textit{times differentiable}; (iii) $h(t)$ \textit{satisfies}
(\ref{eq:Int-152.5}) \textit{and} (\ref{eq:Int-153.5}) \textit{and} (iv) \textit{and as} $t
\rightarrow \infty, t^{-\beta}g^{(j)}(t)=O(t^{-1-\varepsilon})$  for $j=0,
1,..., m$ \textit{and for some} $\varepsilon > 0.$ \textit{Under
these conditions, the result} (\ref{eq:Int-155.5}) \textit{holds with}
\begin{eqnarray}
\label{eq:Int-157.5}
\delta_n(x)=\frac{(-1)^m}{x^m}\int^{\infty}_{0} g^{(m)}_{n}(t)
h^{(-m)}(xt) dt,
\end{eqnarray}
\textit{where} $n$ \textit{is the smallest positive integer such
that} $\lambda+n > m.$
\end{thm}
\begin{proof}
Integrating by part (\ref{eq:Int-156.5})  we get
\begin{eqnarray}
\label{eq:Int-206.5}
\int_0^\infty f_n(t) h(xt)  e^{- \epsilon t^p} dt \nonumber &=&  - \frac{1}{x}\int_0^\infty f_n(t)e^{-\epsilon t^p} h^{(-1)}(xt)dt \\
&& + \; \frac{\epsilon p}{x} \int_0^\infty f_n(t) h^{(-1)}(xt)t^{p-1} e^{-\epsilon t^p} dt
\end{eqnarray}
the integrated term vanishing due to $\rho+\lambda>0$ and condition \textit{(iv)} and the asymptotic behaviour in (\ref{eq:Int-210.5}). The same reasoning, together with Lemma1 and Lemma2, ensures that the second term on the right-hand side of (\ref{eq:Int-206.5}) tends to zero as $\epsilon \rightarrow 0+$.\\
Thus
\begin{eqnarray}
\label{eq:Int-207}
\delta_n(x)=\left(- \frac{1}{x}\right) \lim_{\epsilon \rightarrow 0+} \int_0^\infty f'_n(t) h^{(-1)}(xt) e^{-\epsilon t^p}dt.
\end{eqnarray}
Repeated application of this technique shows that
\begin{eqnarray}
\label{eq:Int-208}
\delta_n(x)\nonumber &=&\frac{(-1)^m}{x^m} \lim_{\epsilon \rightarrow 0+} \int_0^\infty f^{(m)}_n(t) h^{(-m)}(xt) e^{-\epsilon t^p}dt\\ &=&  \frac{(-1)^m}{x^m} \int_0^\infty f^{(m)}_n(t) h^{(-m)}(xt) e^{-\epsilon t^p}dt.
\end{eqnarray}
The last equality again follows from Lemma1.
\end{proof}
\vspace{.5pc} The aim of the present paper is to derive asymptotic
expansion of the wavelet transform given by (\ref{eq:Int-510.3}) for large values
of $a$, using formula (\ref{eq:Int-155.5}). We also obtain asymptotic expansions
for the special transforms corresponding to Morlet wavelet, Mexican
hat wavelet and Haar wavelet.\\
\section{Asymptotic expansion for large $a$}
In this section using aforesaid technique, we obtain asymptotic
expansion of $(W_\psi f)(b, a)$ for large values of $a$, keeping $b$
fixed. We have
\begin{eqnarray}
\label{eq:Int-158.5}
 (W_\psi f)(b, a)\nonumber &=&\frac{\sqrt{a}}{2\pi}\int^{\infty}_{-\infty}
    e^{ib \omega} \overline{\hat{\psi}}(a\omega) \hat{f}(\omega)
    d\omega \\ \nonumber & = & \frac{\sqrt{a}}{2\pi}\biggr\{\int^{\infty}_{0}
    e^{ib \omega} \overline{\hat{\psi}}(a\omega) \hat{f}(\omega)
    d\omega \\ \nonumber &&
+\; \int^{\infty}_{0}
    e^{-ib \omega} \overline{\hat{\psi}}(-a\omega) \hat{f}(-\omega)
    d\omega \biggr\} \\ &=& \frac{\sqrt{a}}{2\pi} (I_1+I_2), \; \textrm{say}.
\end{eqnarray}
Let us set
\begin{eqnarray}
\label{eq:Int-159.5}
h(\omega)=\overline{\hat{\psi}}(\omega).
\end{eqnarray}
Assume that
\begin{eqnarray}
\label{eq:Int-160.5}
\overline{\hat{\psi}}(\omega) \sim \exp(i \tau \omega^p)
\sum^{\infty}_{\tau=0}b_r \omega^{-r-\beta},\;\;\; \beta>0,\;
\omega\rightarrow +\infty, \; \tau \neq 0, \; p \geq 1,
\end{eqnarray}
and
\begin{eqnarray}
\label{eq:Int-161.5}
\hat{f}(\omega)\sim\sum^{\infty}_{s=0} c_s \omega^{s+\lambda-1}\;\;\;
\textrm{as} \;\; \omega \rightarrow 0.
\end{eqnarray}
where $0<\lambda \leq 1$. Also assume that as $ \omega \rightarrow 0, $
\begin{eqnarray}
\label{eq:Int-162.5}
h(\omega) = \overline{\hat{\psi}}(\omega)=O(\omega^\rho), \;\;\;
\rho+\lambda>0.
\end{eqnarray}
Then, as $ \omega \rightarrow 0 $,
\begin{align}
\label{eq:Int-165.5}
    g(\omega)\nonumber &:= e^{ib \omega}\hat{f}(\omega)\\
     \nonumber & \sim
    \sum^{\infty}_{s=0}c_s
    \omega^{s+\lambda-1}\sum^{\infty}_{r=0}\frac{(ib\omega)^r}{r!}\\ \nonumber
    &=\sum^{\infty}_{s=0}\sum^{\infty}_{r=0}c_s\frac{(ib)^r}{r!}\omega^{s+\lambda-1+r}\\
   \nonumber  &=\sum^{\infty}_{s=0}\left\{\sum^{s}_{r=0}\frac{(ib)^r}{r!}c_{s-r}\right\}
    \omega^{s+\lambda-1}\\
    &=\sum^{\infty}_{s=0} d_s\omega^{s+\lambda-1},
\end{align}
where
\begin{eqnarray}
\label{eq:Int-163.5}
d_s =\sum^{s}_{r=0}\frac{(ib)^r}{r!}c_{s-r}.
\end{eqnarray}
For each $n\geq 1,$ we write
\begin{eqnarray}
\label{eq:Int-164.5}
g(\omega)=\sum^{n-1}_{s=0}d_s \omega^{s+\lambda-1}+g_n(\omega).
\end{eqnarray}
The generalized Mellin transform of $h$ is defined by
\begin{eqnarray}
\label{eq:Int-166.5}
 M[h; z_1] =  \lim_{\varepsilon \rightarrow0+}\int^\infty_0\omega^{z_1-1}h(\omega) e^{-\varepsilon
    \omega}d\omega
\end{eqnarray}
Then by (\ref{eq:Int-155.5}),
\begin{eqnarray}
\label{eq:Int-170.5}
I_1(a)=\sum^{n-1}_{s=0}d_s M[h;
s+\lambda]a^{-s-\lambda}+\delta^{1}_{n}(a),
\end{eqnarray}
where
\begin{eqnarray}
\label{eq:Int-167.5}
\delta^{1}_{n}(a)=\lim_{\varepsilon \rightarrow
    0+}\int^{\infty}_{0}g_n(\omega) h(a\omega) e^{-\varepsilon
    \omega}d\omega.
\end{eqnarray}
Also, from (\ref{eq:Int-162.5})  we have
\begin{eqnarray}
\label{eq:Int-168.5}
h(-\omega)=O(\omega^\rho), \;\;\; \omega \rightarrow 0,\;\;
\rho+\lambda>0
\end{eqnarray}
and
\begin{eqnarray}
\label{eq:Int-169.5}
M[h(-\omega);z_1]=\lim_{\varepsilon \rightarrow
    0+}\int^{\infty}_{0} \omega^{z_1-1}h(-\omega) e^{-\varepsilon
    \omega}d\omega.
\end{eqnarray}
Hence
\begin{eqnarray}
\label{eq:Int-171.5}
I_2(a)=\sum^{n-1}_{s=0}d_s (-1)^{s+\lambda+1}M[h(-\omega);
s+\lambda]a^{-s-\lambda}+\delta^{2}_{n}(a),
\end{eqnarray}
where
\begin{eqnarray}
\label{eq:Int-172.5}
\delta^{2}_{n}(a)=\lim_{\varepsilon \rightarrow
    0+}\int^{\infty}_{0}g_n(-\omega) h(-a\omega) e^{-\varepsilon
    \omega}d\omega.
    \end{eqnarray}
Finally, from (\ref{eq:Int-158.5}),(\ref{eq:Int-170.5}) and (\ref{eq:Int-171.5}) we get the
asymptotic expansion:
\begin{eqnarray}
\label{eq:Int-174.5}
(W_\psi f)(b,a) \nonumber
   & =&\frac{\sqrt{a}}{2\pi}\biggr\{\sum^{n-1}_{s=0}d_s\biggr(M\left[\overline{\hat{\psi}}(\omega);
    s+\lambda\right] +(-1)^{s+\lambda+1}\\ && \times \; M\left[\overline{\hat{\psi}}(-\omega)
    s+\lambda\right]\biggr) a^{-s-\lambda} +  \delta_n(a)\biggr\},
\end{eqnarray}
where
\begin{eqnarray}
\label{eq:Int-175.5}
\delta_n(a)\nonumber &=&\lim_{\varepsilon \rightarrow
    0+}\biggr(\int^{\infty}_{0}g_n(\omega) h(a\omega) e^{-\varepsilon
    \omega}d\omega\\ && + \;\int^{\infty}_{0}g_n(-\omega) h(-a\omega) e^{-\varepsilon
    \omega}d\omega\biggr).
\end{eqnarray}
Since $g(\omega) = e^{ib\omega} \hat{f}(\omega),$ the continuity of
$\hat{f}^{(m)}(\omega)$ implies continuity of $g^{(m)}(\omega)$.
Using Theorem\ref{thm:ch1sec2-2.4} we get the following existence theorem for
formula (\ref{eq:Int-175.5}).
\begin{thm}
\label{thm:ch1sec2-2.5}
Assume that (i)
$\hat{f}^{(m)}(\omega)$ is continuous on $(-\infty,
\infty),$ where $m$ is a nonnegative integer;(ii)
$\hat{f}(\omega)$ has asymptotic expansion of the form
(\ref{eq:Int-161.5}) and the expansion is $ m $ times
differentiable,(iii) $\overline{\hat{\psi}}(\omega)$ satisfies
(\ref{eq:Int-159.5}) and (\ref{eq:Int-160.5}) and (iv) as
$\omega\rightarrow \infty,
\omega^{-\beta}\hat{f}^{(j)}(\omega)=O(\omega^{-1-\varepsilon})$
for $j=0, 1, 2,..., m$ and for some $\varepsilon
>0.$ Under these conditions, the result (\ref{eq:Int-174.5}) holds with
\begin{eqnarray}
\label{eq:Int-178.5}
\delta_n(a)=\frac{(-1)^m}{a^m}\int^{\infty}_{-\infty}g^{(m)}_{n}(\omega)(\overline{\hat{\psi}(a\omega)})^{(-m)}d\omega,
\end{eqnarray}
where $n$ is the smallest positive integer such
that $\lambda+n>m.$
\end{thm}
\vspace{.5pc} In the following sections we shall obtain asymptotic
expansions for certain special cases of the general wavelet
transform.
 \section{MORLET WAVELET TRANSFORM}
In this section we choose
\begin{eqnarray*}
\psi(t)=e^{i\omega_0 t-t^2/2}.
\end{eqnarray*}
Then from ~\cite[p. 373]{Debnath} we have
\begin{eqnarray*}
\hat{\psi}(\omega)=\sqrt{2\pi}e^{\frac{-(\omega-\omega_0)^2}{2}},
\end{eqnarray*}
which is exponentially decreasing. Therefore, Theorem\ref{thm:ch1sec2-2.4} is not
directly applicable, but a slight modification of the technique
works well.
Assume that $\hat{f}$ has an asymptotic expansion of the form
(\ref{eq:Int-161.5}). In this case we have
\begin{eqnarray}
\label{eq:Int-180.5}
    h(\omega)\nonumber &= &\overline{\hat{\psi}}(\omega) \\
    &=&\sqrt{2\pi}e^{\frac{-(\omega-\omega_0)^2}{2}}
\end{eqnarray}
and
\begin{eqnarray}
\label{eq:Int-181.5}
h(\omega)=O(1)\;\;\;\;\textrm{as}\;\;\; \omega \rightarrow 0.
\end{eqnarray}
Then from (\ref{eq:Int-170.5}) and (\ref{eq:Int-180.5}), we get
\begin{eqnarray}
\label{eq:Int-182.5}
I_1(a) \nonumber &=& \sum^{n-1}_{s=0}d_s
M\left[\sqrt{2\pi}e^{\frac{-(\omega-\omega_0)^2}{2}};
s+\lambda\right]a^{-s-\lambda}\\
&&+\; \lim_{\varepsilon\rightarrow
0+}\int^{\infty}_{0}g_n(\omega)\sqrt{2\pi}e^{\frac{-(a\omega-\omega_0)^2}{2}}e^{-\varepsilon
\omega}d\omega,
\end{eqnarray}
where
\begin{eqnarray*}
M\left[\sqrt{2\pi}e^{\frac{-(\omega-\omega_0)^2}{2}};
s+\lambda\right]&=&\sqrt{2\pi}\int^{\infty}_{0}\omega^{s+\lambda-1}e^{\frac{-(\omega-\omega_0)^2}{2}}d\omega\\
&=& \sqrt{2\pi} e^{-\frac{\omega^2_0}{2}}
\int^{\infty}_{0}\omega^{s+\lambda-1}
e^{-\frac{\omega^2}{2}+\omega\omega_0}d\omega.
\end{eqnarray*}
Evaluating the last integral by means of formula ~\cite[(31), p.320]{Erdelyi}:
\begin{eqnarray*}
\int^{\infty}_{0}x^{s-1} e^{-(x^2/2)-\beta x}dx=
e^{(\beta^2/4)}\Gamma(s) D_{-s}(\beta), \;\;\; Re(s)>0,
\end{eqnarray*}
where $D_{-\nu}(x)$ denotes parabolic cylinder function, we get
\begin{eqnarray}
\label{eq:Int-184.5}
M\left[\sqrt{2\pi}e^{\frac{-(\omega-\omega_0)^2}{2}};
s+\lambda\right]=\sqrt{2\pi}e^{\frac{-\omega^2_0}{4}}\Gamma(s+\lambda)D_{-(s+\lambda)}(-\omega_0),
\;\;\; s+\lambda>0.
\end{eqnarray}
From (\ref{eq:Int-182.5}) and (\ref{eq:Int-184.5}), we get
\begin{eqnarray}
\label{eq:Int-185.5}
    I_1(a)\nonumber &=&\sqrt{2\pi}e^{\frac{-\omega^2_0}{4}}\sum^{n-1}_{s=0}d_s
    \Gamma(s+\lambda)D_{-(s+\lambda)}(-\omega_0) \,\, a^{-s-\lambda}\\
    &&+ \;
    \int^{\infty}_{0}g_n(\omega)\sqrt{2\pi}e^{\frac{-(a\omega-\omega_0)^2}{2}}d\omega.
    \end{eqnarray}
Similarly, we get
\begin{eqnarray}
\label{eq:Int-186.5}
    I_2(a)\nonumber&=&\sqrt{2\pi}e^{\frac{-\omega^2_0}{4}}\sum^{n-1}_{s=0}d_s
    \Gamma(s+\lambda)(-1)^{s+\lambda-1}D_{-(s+\lambda)}(\omega_0)\,\, a^{-s-\lambda}\\
    && + \; \int^{\infty}_{0}g_n(-\omega)\sqrt{2\pi}e^{-\frac{(a\omega+\omega_0)^2}{2}}d\omega.
    \end{eqnarray}
Finally, using (\ref{eq:Int-158.5}), (\ref{eq:Int-185.5}) and (\ref{eq:Int-186.5}) we get
\begin{eqnarray}
\label{eq:Int-187.5}
(W_\psi f)(b, a)\nonumber &=&e^{\frac{-\omega^2_0}{4}}\sum^{n-1}_{s=0}d_s
    \Gamma(s+\lambda)\left[D_{-(s+\lambda)}(-\omega_0)\right.\\
    && + \left. \; (-1)^{s+\lambda-1}D_{-(s+\lambda)}(\omega_0)\right]a^{-s-\lambda+\frac{1}{2}}+\delta_n(a),
   \end{eqnarray}
where
\begin{eqnarray*}
\delta_n(a)=\sqrt{a}\int^{\infty}_{0}g_n(\omega)e^{-(a\omega-\omega_0)^2}d\omega+\sqrt{a}\int^{\infty}_{0}g_n(-\omega)e^{-(a\omega+\omega_0)^2}d\omega.
\end{eqnarray*}
Using Theorem \ref{thm:ch1sec2-2.5} we get the following existence theorem for
formula (\ref{eq:Int-187.5}).
\begin{thm}
\label{thm:ch1sec3-3.1}
Assume that $\hat{f}(\omega)$ satisfies conditions of Theorem \ref{thm:ch1sec2-2.5}.
Then the result (\ref{eq:Int-187.5}) holds with
\begin{eqnarray*}
\delta_n(a)=(-1)^m
a^{-m+1/2}\int^{\infty}_{-\infty}g^{(m)}_{n}(\omega)\left(
e^{\frac{-(a\omega-\omega_0)^2}{2}}\right)^{(-m)}d\omega,
\end{eqnarray*}
where $n$ is the smallest positive integer such
that $\lambda+n>m.$
\end{thm}
\section{MEXICAN HAT WAVELET TRANSFORM}
In this section we choose
\begin{eqnarray*}
\psi(t)=(1-t^2)e^{-t^2/2}.
\end{eqnarray*}
Then from ~\cite[p.372]{Debnath}
\begin{eqnarray}
\label{eq:Int-188.5}
h(\omega):=\hat{\psi}(\omega)=\sqrt{2\pi}\omega^2 e^{-\omega^2/2};
\end{eqnarray}
so that
\begin{eqnarray}
\label{eq:Int-189.5}
h(\omega)=O(\omega^2), \;\;\; \omega \rightarrow 0.
\end{eqnarray}
Assume that $\hat{f}$ has an asymptotic expansion of the form
(\ref{eq:Int-161.5}), and satisfies
\begin{eqnarray}
\label{eq:Int-190.5}
\hat{f}(\omega)=O(e^{\sigma \omega^2}), \;\;\;\omega \rightarrow
+\infty,
\end{eqnarray}
for some $\sigma >0$. Therefore,
\begin{eqnarray}
\label{eq:Int-191.5}
g(\omega):= e^{ib \omega} \hat{f}(\omega)=O(e^{\sigma \omega^2}),
\;\;\;\omega \rightarrow +\infty.
\end{eqnarray}
Then by (\ref{eq:Int-166.5}) and (\ref{eq:Int-188.5}), we get
\begin{eqnarray}
\label{eq:Int-192.5}
I_1(a)=\sum^{n-1}_{s=0}d_s M [\sqrt{2\pi}\omega^2e^{-\omega^2/2};
s+\lambda]a^{-s-\lambda}+\delta^{1}_{n}(a),
\end{eqnarray}
where
\begin{align*}
M [\sqrt{2\pi}\omega^2e^{-\omega^2/2};
s+\lambda]&=\sqrt{2\pi}\int^{\infty}_{0}\omega^{s+\lambda+1}e^{-\omega^2/2}d\omega\\
&=\sqrt{\pi}\,\, 2^{(s+\lambda+1)/2}\,\, \Gamma\left(\frac{s+\lambda+2}{2}\right),
\end{align*}
and
\begin{eqnarray}
\label{eq:Int-193.5}
\delta^{1}_{n}=\int^{\infty}_{0}g_n(\omega)\sqrt{2\pi}(a\omega)^2e^{-(a\omega)^2/2}d\omega.
\end{eqnarray}
Similarly, we get
\begin{eqnarray}
\label{eq:Int-194.5}
I_2(a)\nonumber &=&\sqrt{\pi}2^{(\lambda+1)/2}\sum^{n-1}_{s=0}d_s(-1)^{s+\lambda-1}2^{s/2}\Gamma\left(\frac{s+\lambda+2}{2}\right)a^{-s-\lambda}\\ &&+ \; \delta^{2}_{n}(a),
\end{eqnarray}
where
\begin{eqnarray}
\label{eq:Int-195.5}
\delta^{2}_{n}(a)=\int^{\infty}_{0}g_n(-\omega)\sqrt{2\pi}(a\omega)^2e^{-(a\omega)^2/2}d\omega.
\end{eqnarray}
Finally, using (\ref{eq:Int-158.5}), (\ref{eq:Int-192.5}) and (\ref{eq:Int-194.5}), we have
\begin{eqnarray}
\label{eq:Int-196.5}
(W_\psi f)(b,
a)\nonumber &=&\frac{2^{(\lambda+1)/2}}{\sqrt{\pi}}\sum^{n-1}_{s=0}d_s
2^{s/2}\Gamma\left(\frac{s+\lambda+2}{2}\right)\{1+(-1)^{s+\lambda-1}\}\\ && \times \; a^{-s-\lambda+1/2}
 +  \delta_n(a),
\end{eqnarray}
where
\begin{eqnarray*}
\delta_{n}(a)=2^{3/2}\sqrt{\pi}\int^{\infty}_{0}g_n(\omega)(a\omega)^2e^{-(a\omega)^2/2}d\omega.
\end{eqnarray*}
Existence theorem for formula (\ref{eq:Int-196.5}) is as follows:
\begin{thm}
\label{thm:ch1sec4-4.1}
Assume that $\hat{f}(\omega)$ satisfies conditions of Theorem\ref{thm:ch1sec2-2.5}
Then the result (\ref{eq:Int-196.5}) holds with
\begin{eqnarray*}
\delta_{n}(a)=\left(\frac{-1}{a}\right)^m
2^{3/2}\sqrt{\pi}\int^{\infty}_{-\infty}g^{(m)}_{n}(\omega)((a\omega)^2e^{-(a\omega)^2/2})^{(-m)}d\omega,
\end{eqnarray*}
where $n$ is the smallest positive integer such
that $\lambda+n>m.$
\end{thm}

\section{ Haar wavelet transform}
In this section we choose
\begin{eqnarray*}
\psi(t)=\left\{\begin{array}{cc}
                 1, & 0\leq t<1/2 \\
                 -1, & 1/2\leq t <1 \\
                 0, & \textrm{otherwise}.
               \end{array}\right.
\end{eqnarray*}
Then from ~\cite[p.368]{Debnath},
\begin{eqnarray}
\label{eq:Int-199.5}
\overline{\hat{\psi}}(\omega)=4i
e^{-i\omega/2}\frac{sin^2\omega/4}{\omega}=\frac{i}{\omega}(1-2e^{i\omega/2}+e^{i\omega}).
\end{eqnarray}
Although the condition $\beta >0$ of (\ref{eq:Int-159.5}) is not satisfied in this case but
the result  (\ref{eq:Int-153.5})- (\ref{eq:Int-154.5}) remains valid, cf.~\cite[p.753]{Wong}.

\vspace{.5pc} Clearly,
\begin{eqnarray}
\label{eq:Int-200.5}
h(\omega)=O(\omega), \;\;\; as \;\; \omega \rightarrow
0.
\end{eqnarray}
Assume that $\hat{f}(\omega)$ has an asymptotic expansion of the
form (\ref{eq:Int-161.5}).
\vspace{.5pc} Using (\ref{eq:Int-158.5}) and (\ref{eq:Int-199.5}) we get
\begin{eqnarray*}
    I_1(a)&=&\int^{\infty}_{0}e^{ib\omega}\hat{f}(\omega)\frac{1}{a\omega}(1-2e^{ia\omega/2}+e^{ia\omega})d\omega \\
    &=&\frac{i}{a}F(b)-2i\int^{\infty}_{0}e^{ib\omega}\hat{f}(\omega)\frac{e^{ia\omega/2}}{a\omega}d\omega\\ && + \;i\int^{\infty}_{0}e^{ib\omega}\hat{f}(\omega)\frac{e^{ia\omega}}{a\omega}d\omega,
\end{eqnarray*}
where
\begin{eqnarray*}
F(b)=\int^{\infty}_{0}e^{ib\omega}\frac{\hat{f}(\omega)}{\omega}d\omega.
\end{eqnarray*}
Then, (\ref{eq:Int-165.5}) and the generalized Mellin transform formula ~\cite[Lemma 2, p.198]{Rwong}:
\begin{eqnarray*}
M[e^{it}; z]=e^{i\pi z/2}\Gamma(z)
\end{eqnarray*}
we get
\begin{eqnarray*}
I_1(a) &=&\frac{i}{a}F(b)-\frac{2i}{a}\int^{\infty}_{0}\left[\sum^{\infty}_{s=1}d_s
    \omega^{s+\lambda-2}+g_n(\omega)\right]e^{ia\omega/2}d\omega \\ && + \; \frac{i}{a}\int^{\infty}_{0}\left[\sum^{\infty}_{s=1}d_s
    \omega^{s+\lambda-2}+g_n(\omega)\right]e^{ia\omega}d\omega \\
    &=& \frac{i}{a}F(b)+\frac{2i}{a}\biggr\{\sum^{n-1}_{s=1}d_s\Gamma(s+\lambda-1)(a/2)^{-s-\lambda+1}
    e^{i\pi(s+\lambda)/2}.\\
   \nonumber  && - \;(2i/a)^n\int^{\infty}_{0}g^{(n)}_{n}(\omega)e^{ia\omega/2}d\omega\biggr\}
\end{eqnarray*}
\begin{eqnarray}
\label{eq:Int-201.5}
     \nonumber  && + \;\frac{i}{a}\biggr\{-\sum^{n-1}_{s=1}d_s\Gamma(s+\lambda-1)(a)^{-s-\lambda+1}
    e^{i\pi(s+\lambda)/2} \nonumber \\
    && + \;(i/a)^n\int^{\infty}_{0}g^{(n)}_{n}(\omega)e^{ia\omega}d\omega\biggr\}\nonumber \\&=& \frac{i}{a}F(b)+i\sum^{n-1}_{s=1}d_s\Gamma(s+\lambda-1)a^{-s-\lambda}(2^{s+\lambda}-1)e^{i\pi(s+\lambda)/2} \nonumber \\
    && + \; (i/a)^{n+1}\int^{\infty}_{0}g^{(n)}_{n}(\omega)(e^{ia\omega}-2^{n+1}e^{ia\omega/2})d\omega.
\end{eqnarray}
Notice that for existence of the Mellin transform in the above case
we have to assume that $d_0=0.$
\vspace{.5pc} Similarly,
\begin{eqnarray}
\label{eq:Int-202.5}
I_2(a)\nonumber &=&\frac{i}{a}\int^{0}_{-\infty}e^{ib\omega}\frac{\hat{f}(\omega)}{\omega}d\omega
\nonumber  +  i\sum^{n-1}_{s=1}d_s\Gamma(s+\lambda-1)a^{-s-\lambda}(-1)^{s+\lambda-1}  \\ && \times \;(2^{s+\lambda}-1) e^{i\pi(s+\lambda)/2}
 +  (i/a)^{n+1}\int^{\infty}_{0}g^{(n)}_{n}(-\omega)\nonumber
 \\ && \times \;(e^{-ia\omega}-2^{n+1}e^{-ia\omega/2})d\omega.
 \end{eqnarray}
Finally,using formula ~\cite[(15), p.152]{Erdelyi}, from (\ref{eq:Int-158.5}),(\ref{eq:Int-201.5}) and (\ref{eq:Int-202.5}) we get
\begin{eqnarray}
\label{eq:Int-203.5}
(W_\psi f)(b, a)\nonumber &=&\frac{i}{\sqrt{a}}f^{(-1)}(b) +\frac{i}{\pi}\sum^{n-1}_{s=0}d_s
\Gamma(s+\lambda-1)a^{-s-\lambda+1/2}\\ && \times \{1+(-1)^{s+\lambda-1}\}(2^{s+\lambda}-1)e^{i\pi(s+\lambda)/2}+  \delta_n(a),
\end{eqnarray}
where
\begin{eqnarray*}
f^{(-1)}(b)=(D^{-1}f)(b)
\end{eqnarray*}
and
\begin{eqnarray}
\label{eq:Int-204.5}
\delta_n(a)=(i/a)^{n+1}\frac{\sqrt{a}}{2\pi}\int^{\infty}_{0}g^{(n)}_{n}(\omega)(e^{ia\omega}-2^{n+1}e^{ia\omega/2}d\omega.
\end{eqnarray}
Existence theorem for (\ref{eq:Int-203.5}) is as follows:
\begin{thm}
\label{thm:ch1sec5-5.1}
Assume that $\hat{f}(\omega)$ satisfies conditions of Theorem\ref{thm:ch1sec2-2.5}.
Then the result(\ref{eq:Int-203.5}) holds with
\[
\delta_n(a)=(i/a)^{m+1}\frac{\sqrt{a}}{2\pi}\int^{\infty}_{-\infty}g^{(m)}_{n}(\omega)(e^{ia\omega}-2^{m+1}e^{ia\omega/2})d\omega,
\]
\textit{where} $n$ is the smallest positive integer such
that $\lambda+n>m.$
\end{thm}
\begin{flushleft}
R S Pathak \\
DST Center for Interdisciplinary Mathematical Sciences \\
Banaras Hindu University \\
Varanasi-221005,India \\
e-mail: ranshankarpathak@yahoo.com \\
\vspace*{0.5cm}
Ashish Pathak \\
Department of Mathemtics and Statistics
\\  Dr. Harisingh Gour Central University
\\   Sagar-470003, India.
\\  e-mail: ashishpathak@dhsgsu.ac.in \\  \
   \end{flushleft}
\end{document}